\documentclass{agtart_a}
\pdfoutput=1

\title[Exotic relation modules and homotopy types for certain 
       1--relator groups]{Exotic relation modules and homotopy types 
       \\for certain 1--relator groups}                    

\author{Jens Harlander}
\givenname{Jens}
\surname{Harlander}
\address{Department of Mathematics\\Western Kentucky
         University\\\newline
         Bowling Green, KY 42101\\USA}
\email{jens.harlander@wku.edu}
\urladdr{}

\author{Jacqueline A Jensen}
\givenname{Jacqueline A}
\surname{Jensen}
\address{Department of Mathematics and Statistics\\
         Sam Houston State University\\\newline
         Huntsville, TX 77341\\USA}
\email{jensen@shsu.edu}
\urladdr{}

\volumenumber{6}
\issuenumber{}
\publicationyear{2006}
\papernumber{75}
\startpage{2163}
\endpage{2173}

\doi{}
\MR{}
\Zbl{}

\keyword{2-dimensional complex}
\keyword{homotopy-type}
\keyword{stably free modules}
\subject{primary}{msc2000}{57M20}
\subject{secondary}{msc2000}{57M05}

\received{18 May 2006}
\revised{}
\accepted{25 September 2006}
\published{19 November 2006}
\publishedonline{19 November 2006}
\proposed{}
\seconded{}
\corresponding{}
\editor{}
\version{}

\arxivreference{}

\AtBeginDocument{\let\bar\wbar\let\tilde\wtilde\let\hat\what}

\makeatletter
\def\cnewtheorem#1[#2]#3{\newtheorem{#1}{#3}[section]
\expandafter\let\csname c@#1\endcsname\c@thm}
\makeatother

\newtheorem{thm}{Theorem}[section]   
\cnewtheorem{lemma}[thm]{Lemma}       
\cnewtheorem{corollary}[thm]{Corollary}
\cnewtheorem{proposition}[thm]{Proposition}

\theoremstyle{definition}
\cnewtheorem{defn}[thm]{Definition}

\def\<{\langle}
\def\>{\rangle}

\,
\,
\,

\begin{document}

\begin{asciiabstract}
Using stably free non-free relation modules we construct an infinite
collection of 2-dimensional homotopy types, each of
Euler-characteristic one and with trefoil fundamental group. This
provides an affirmative answer to a question asked by Berridge and
Dunwoody [J. London Math. Soc. 19 (1979) 433-436].  We also give new
examples of exotic relation modules. We show that the relation module
associated with the generating set x, y^4 for the Baumslag-Solitar
group <x, y | xy^2x^{-1}=y^3> is stably free non-free of rank one.
\end{asciiabstract}

\begin{htmlabstract}
Using stably free non-free relation modules we construct an infinite
collection of 2--dimensional homotopy types, each of
Euler-characteristic one and with trefoil fundamental group. This
provides an affirmative answer to a question asked by Berridge and
Dunwoody [J. London Math. Soc. 19 (1979) 433--436]. We also give new
examples of exotic relation modules. We show that the relation module
associated with the generating set { x, y<sup>4</sup>} for the
Baumslag--Solitar group < x, y |
xy<sup>2</sup>x<sup>-1</sup>=y<sup>3</sup> > is stably free
non-free of rank one.
\end{htmlabstract}

\begin{abstract} 
Using stably free non-free relation modules we construct
an infinite collection of 2--dimensional homotopy types, each of
Euler-characteristic one and with trefoil fundamental group. This
provides an affirmative answer to a question asked by Berridge and
Dunwoody \cite{BerridgeDunwoody}. We also give new examples of
exotic relation modules. We show that the relation module
associated with the generating set $\{ x, y^4\}$ for the
Baumslag--Solitar group $\langle x, y \ \vert\
xy^2x^{-1}=y^3\rangle$ is stably free non-free of rank one.
\end{abstract}

\maketitle

\section{Introduction}

Given a group $G$ and an integer $n$ the homotopy classification
program in dimension two aims to determine all 2--complexes (up to
homotopy) with fundamental group $G$ and Euler-characteristic $n$ (see
Dunwoody \cite{Dunwoody1}, Harlander and Jensen
\cite{HarlanderJensen}, Hog-Angeloni, Metzler and Sieradski
\cite{metzlerbook}, Beyl and Waller \cite{Beyl,BeylWaller}, Dyer and
Sieradski \cite{Dyer2,Dyer1}, Jensen \cite{jensenthesis} and Johnson
\cite{Johnson}). If $K$ is a $2$--complex then it is not difficult to
see that the Euler-characteristic $\chi(K)$ is bounded from below by
$\sum_{i=0}^2 (-1)^{i}\textrm{dim} H_i(G,\mathbb{Q})$, a constant that
only depends on the homology of $G$. Thus we can define
$\chi_{\min}(G)$ to be the minimal Euler-characteristic that can occur
for a finite 2--complex with fundamental group $G$. If $G$ is a group
of finite geometric dimension 2, that is $G$ is the fundamental group
of a finite aspherical 2--complex $K$, then $\chi_{min}(G)=\chi(K)$
and $K$ is, up to homotopy, the unique 2--complex on the minimal
Euler-characteristic level. We show that if $G$ is of geometric
dimension 2 and admits a stably free non-free relation module of rank
$k$, then there are at least two homotopically distinct 2--complexes
with the same Euler-characteristic, $\chi_{min}(G)+k$
(\fullref{cor4_2} and \fullref{thm4_3}). We use this to construct an
infinite collection of homotopically distinct 2--complexes for the
trefoil group, all of Euler characteristic one (\fullref{thm4_4}).
This provides an affirmative answer to a question raised by Berridge
and Dunwoody \cite{BerridgeDunwoody} (see also Lustig \cite{Lustig}
for closely related results).

These topological applications rely on the existence of
non-free relation modules for groups of geometric dimension 2. Let
$G$ be a group and ${\bf x}$ be a generating set for $G$. Let $F$
be the free group with basis in one-to-one correspondence to ${\bf
x}$. The kernel of the canonical map $F\rightarrow G$ is denoted
by $R(G,{\bf x})$ and is called the relation subgroup associated
with ${\bf x}$. If we abelianize $R=R(G,{\bf x})$ we obtain a
$\mathbb Z G$--module $M(G,{\bf x})=R/[R,R]$,where the $G$--action
is given by conjugation. This module is called the relation module
associated with ${\bf x}$.

Consider the group $G$ presented by $\langle x, y  \ |\
xy^2x^{-1}=y^3 \rangle$. Observe that the elements $x$ and $z=y^4$
also generate $G$. Indeed, since $xzx^{-1}=y^6$, the element $y^2$
is in $\langle x,z \rangle$. Since $xy^2x^{-1}=y^3$, we see that
$y^3$ is in $\langle x,z \rangle$ and hence so is $y$.

Graham Higman observed (see Lyndon and Schupp \cite[page
93]{LyndonSchupp}) that the generating set consisting of $x$ and $z$
does not support a 1--relator presentation for $G$. We prove a stronger
result.

\begin{thm}\label{thm1_1} Let $G$ be the group defined by $\langle x, y \ |\
xy^2x^{-1}=y^3 \rangle$ and let $z=y^4$. Then the relation module
$M(G,\{x,z\})$ cannot be generated by a single element.
\end{thm}

\begin{corollary}\label{cor1_2} The relation module $M=M(G,\{ x, z \})$ is
stably free non-free of rank one: $M\oplus \mathbb Z G \approx
\mathbb Z G^2$.
\end{corollary}

\begin{proof} Since $y$ is a redundant generator we have that
$$M(G,\{x,y,z\})\approx M(G,\{x,z\})\oplus \mathbb Z G.$$ Since the
generating set $\{x,y,z\}$ supports the aspherical presentation
$$\langle x,y,z \vert xy^2x^{-1}=y^3, z=y^4 \rangle$$ (see
Lyndon \cite{Lyndon}) it follows that $M(G,\{x,y,z\})\approx \mathbb Z
G^2$. \end{proof}

In \cite{Dunwoody1} Martin Dunwoody shows
analogous results for the trefoil group. Later Berridge and
Dunwoody \cite{BerridgeDunwoody} show that there are infinitely
many non-isomorphic stably free non-free relation modules for the
trefoil group, all of rank one. For related results see also Lewin
\cite{Lewin}.

\section{Some combinatorial group theory}

Let $G$ be the group presented by $\langle x, y \ |\
xy^2x^{-1}=y^3 \rangle$ and let $z=y^4$.

\begin{lemma}\label{lemma2_1}
The kernel $K$ of the epimorphism $$\bar G=\langle x,z \ |\
z=[x,z]^2 \rangle \rightarrow G$$ that sends $x$ to $x$ and $z$ to
$z$ is non-trivial and free.
\end{lemma}

\begin{proof} Let $\bar H$ be the normal closure of $z$ in $\bar G$ and let
$H$ be the normal closure of $z$ in $G$. The epimorphism in the
statement of the lemma restricts to an epimorphism from $\bar H$ to
$H$ with kernel $K$. Indeed, since both $G/H$ and and $\bar G/\bar
H$ are infinite cyclic (generated by $x$), it follows that the
kernel $K\bar H/\bar H$ of the epimorphism $\bar G/\bar H
\rightarrow G/H$ is trivial and so $K$ is contained in $\bar H$. Let
$z_i=x^izx^{-i}$, $i\in \mathbb Z$. Then $\bar H$ has a presentation
$$\bar H=\langle z_i \ |\  z_i=(z_{i+1}z_i^{-1})^2 \rangle, i\in \mathbb
Z.$$ If we define $u_i=z_{i+1}z_i^{-1}$ then we obtain a
presentation
$$\bar H=\langle z_i, u_i \ |\  u_i=z_{i+1}z_i^{-1}, z_i=(z_{i+1}z_i^{-1})^2 \rangle,
i\in \mathbb Z$$ and hence via Tietze transformations
$$\bar H=\langle u_i \ |\  u_{i+1}^2=u_i^3 \rangle,
i\in \mathbb Z.$$ So we see that $\bar H$ is an amalgamated
product with infinite cyclic vertex groups $\langle u_i \rangle$
and infinite cyclic edge groups that give the relations
$u_{i+1}^2=u_i^3$. Note that the epimorphism $\bar H \rightarrow
H$ sends $u_i$ to $y_i^2=x^iy^2x^{-i}$. Since $[y^2,xy^2x^{-1}]=1$
in $G$, we see that $K$ contains the element $[u_0,u_1]$, which by
the Normal Form Theorem for amalgamated products is non-trivial.
It remains to be shown that $K$ is free. Since $\bar H$ acts on a
tree with infinite cyclic vertex stabilizers conjugate to $\langle
u_i \rangle$ and $\langle u_i\rangle K/K$ is the infinite cyclic
subgroup $\langle y_i^2 \rangle$ of $H$, it follows that $K$
intersects the conjugates of $\langle u_i \rangle$ trivially.
Hence $K$ acts freely on a tree and hence is free.\end{proof}

\begin{lemma}\label{lemma2_2}
Let ${\bf s}(z,xzx^{-1})$ be a set of relations among $z$ and
$xzx^{-1}$ that holds among the generators $x$ and $z$ of $G$. Let
$L$ be the kernel of an epimorphism $\bar G=\langle x,z \ |\  {\bf
s}(z,xzx^{-1})\rangle \rightarrow G$ that sends $x$ to $x$ and $z$
to $z$. Then $L$ is not perfect (that is $L/[L,L]\ne 0$).
\end{lemma}

\begin{proof} Let $\bar H$ and $H$ be the normal closures of $z$ in $\bar
G$ and $G$, respectively. As in the previous lemma $L$ is also in the
kernel of the restriction of the epimorphism to $\bar H$. The
group $\bar H$ has a staggered presentation $$\bar H=\langle z_i \
|\  {\bf s_i}(z_i,z_{i+1})\rangle, i\in \mathbb Z$$ where ${\bf
s_i}={\bf s_i}(z_i,z_{i+1})$ is a set of relations that hold among
the two elements $z_i, z_{i+1}$ of $H$. Let $F$ be the free group
on the $z_i$, $i\in \mathbb Z$, and let $S$ be the normal closure
of the $\bigcup_{i\in\mathbb Z}{\bf s_i}$, in $F$. Then $\bar
H=F/S$ and $L=J/S$ for some normal subgroup $J$ of $F$. Notice
that $H=F/J$ and that $J$ contains the elements
$c_i=[z_i,z_{i+1}]$ and $d_i=z_{i}^{-3}z_{i+1}^2$, $i\in \mathbb
Z$ (because they present the trivial element in $H$). We claim
that the normal closure of the set $\bigcup_{i\in \mathbb Z}{\bf
s_i}\cup\{c_i, d_i, i\in \mathbb Z \}$ in $F$ is the same as the
normal closure of $\{ e_i=z_i^{-1}(z_{i+1}z_i^{-1})^2$, $i\in
\mathbb Z$\} in $F$. To see this it suffices to show that for
every fixed $i\in \mathbb Z$ we have
$$^{F(z_i,z_{i+1})}\langle\langle {\bf s_i}, c_i,
d_i\rangle\rangle= {^{F(z_i,z_{i+1})}\langle\langle} e_i
\rangle\rangle.$$ Clearly the right hand side is contained in the
left hand side because $e_i$ is a product of conjugates of $c_i$ and
$d_i$ (if we are allowed to commute $z_i$ and $z_{i+1}$ then we can
turn $e_i$ into $d_i$). In order to show the other inclusion
consider the epimorphism
$$\langle z_i, z_{i+1} \ |\  e_i \rangle \rightarrow \langle z_i,
z_{i+1} \ |\  {\bf s_i}, c_i, d_i \rangle.$$ We will show that
both groups are infinite cyclic. Hence this epimorphism is an
isomorphism, and that settles the claim. Note first that the group
on the left is infinite cyclic, generated by $z_{i+1}z_i^{-1}$.
Let us consider the group on the right. Notice that the relations
${\bf s_i}$, $c_i$ and $d_i$ hold in $H$ (i.e. these elements are
contained in $J$), ${\bf s_i}$ by hypothesis and $c_i$, $d_i$ by
direct inspection. Thus we have an epimorphism
$$\langle z_i, z_{i+1} \ |\  {\bf s_i}, c_i, d_i \rangle\rightarrow \langle
z_i, z_{i+1}\rangle = H_i$$ onto the subgroup $H_i$ of $H$
generated by $z_i=y_i^4=x^iy^4x^{-i}$,
$z_{i+1}=y_{i+1}^4=x^{i+1}y^4x^{-(i+1)}$. Since in $H$ we have
$y_{i+1}^4=y_i^6$, we see that $\langle z_i,
z_{i+1}\rangle=\langle y_i^4, y_i^6\rangle=\langle y_i^2\rangle$,
which is an infinite cyclic subgroup of $H$. Thus $\langle z_i,
z_{i+1} \ |\  {\bf s_i}, c_i, d_i \rangle$ is the image of an
infinite cyclic group and has an infinite cyclic image. Hence it is
infinite cyclic. This settles the claim.

Let $E$ be the normal closure of $\{ e_i$, $i\in \mathbb Z \}$ in
$F$ and let $K=J/E$. We have just shown that $S\subseteq E$, so
$K$ is a homomorphic image of $L=J/S$. So if we assume that $L$ is
perfect, we conclude that $K$ is perfect as well. But according to
the previous \fullref{lemma2_1} the group $K$ is non-trivial and free. So
$L$ can not be perfect. \end{proof}

\section{Some module theory}

\begin{proposition}\label{prop3_1} Suppose $G$ is a group and $M(G,{\bf x})$ is the
relation module associated with some generating set ${\bf x}$.
Suppose furthermore that ${\bf s}$ is a subset of $R=R(G,{\bf x})$
that gives a set of generators for the $\mathbb Z G$--module
$M(G,{\bf x})$. Let $\tilde G$ be the group defined by the
presentation $\langle {\bf x} \ |\  {\bf s} \rangle$. Then the
kernel $P$ of the natural surjection from $\tilde G$ onto $G$ is
perfect.
\end{proposition}

\begin{proof} Let $F$ be the free group on ${\bf x}$ and let $S$ be the
normal closure of ${\bf s}$ in $F$. Since $s[R,R]$, $s\in {\bf
s}$, generates the relation module we have $R=S[R,R]$. Since
$\tilde G=F/S$ and $G=F/R$, the kernel of the map $\tilde G
\rightarrow G$ is $P=R/S$. Thus $P/[P,P]=R/S[R,R]=0$. \end{proof}

Let $G$ be a group, $F=F({\bf a}\cup{\bf b})$ be a
free group on the union of sets ${\bf a}$ and ${\bf b}$ and let
$\pi\co  F({\bf a}\cup{\bf b})\rightarrow G$ be a group epimorphism.
Let $R=R(G,{\bf a}\cup {\bf b})$ be the kernel of $\pi$. Assume
that $Q=\pi(F({\bf b}))$ is a free group on basis ${\bf b}$ and
let $H$ be the normal closure of $\pi(F({\bf a}))$ in $G$. Note
that $G$ is a semi-direct product $H\rtimes Q$. Let ${\bf {\bar
a}}=\{ faf^{-1} \ |\  f\in F, a\in{\bf a}\}$. Then $F({\bf {\bar
a}})$, the free group on ${\bf {\bar a}}$, is the normal closure
of ${\bf a}$ in $F$. Let $\pi'\co  F({\bf {\bar a}})\rightarrow H$ be
the restriction of $\pi$ and let $S=S(H, {\bf{\bar a}})$ be the
kernel of $\pi'$. Note that since $R\subseteq F({\bf{\bar a}})$ we
have $S=R$. In particular $M(H,{\bf {\bar a}})=M(G,{\bf a}\cup{\bf
b})$ as $\mathbb Z G$--modules, where the $G$--action on $M(H,{\bf
{\bar a}})$ is conjugation: $fR*s[S,S]=fsf^{-1}[S,S]$, $f\in F$,
$s\in S$.

Let ${\bf x}$ be a generating set for the group $G$. The Cayley-graph
$\Gamma(G, {\bf x})$ is a graph with vertex set $G$ and edge set
$G\times {\bf x}$. The initial vertex of the edge $(g,x)$ is $g$, the
terminal vertex is the product $gx$. The group $G$ acts on this graph
via left multiplication and induces $\mathbb Z G$--module structures on
the homology groups. Every element of $R=R(G,{\bf x})$ can be lifted
to a closed edge path in the Cayley-graph and this construction is the
basis for the Fox-derivative ${\cal F}\co M(G,{\bf x})\rightarrow
H_1(\Gamma(G,{\bf x}))$, ${\cal F}(r[R,R])=\sum_{x\in {\bf
x}}\frac{\partial r}{\partial x}e_x$, where $e_x$ denotes the edge
$(1,x)$. The Fox-derivative can be shown to be a $\mathbb Z G$--module
isomorphism between the relation module and the first homology of the
Cayley-graph (see \cite[Chapter II, Section 3]{LyndonSchupp}).

Let us now specialize to the situation where $G=\langle x,y \ |\
xy^2x^{-1}=y^3 \rangle$. Let $z=y^4$, ${\bf a}=\{ z \}$, ${\bf
b}=\{ x \}$. Then ${\bf {\bar a}}=\{ z_i \ |\  i\in \mathbb Z \}$,
where $z_i=x^izx^{-i}$. We have $M(G,\{x,z\})=M(H,\{z_i\}_{i\in
\mathbb Z})$ as $\mathbb Z G$--modules. Now $$H_1(\Gamma(H,\{
z_i\}_{i\in \mathbb Z}))=\textrm{ker} (\bigoplus_{i\in \mathbb Z}\mathbb Z
He_i\stackrel{\partial}{\rightarrow}\mathbb Z H),$$ where $e_i$ is
the edge $(1,z_i)$ and hence $\partial(e_i)=z_i-1$. Since
$G=H\rtimes \langle x \rangle$, the group ring $\mathbb Z G$ is a
skewed Laurent-polynomial ring $\mathbb Z H[x^{\pm 1}]$.
Furthermore $\bigoplus_{i\in \mathbb Z}\mathbb Z He_i$ is
isomorphic to $\mathbb Z G=\mathbb Z H[x^{\pm 1}]$, the
isomorphism sending $e_i$ to $x^i$ and the map $\partial\co 
\bigoplus_{i\in \mathbb Z}\mathbb Z He_i\rightarrow \mathbb Z H$,
$\partial(x^i)=z_i-1$, is a $\mathbb Z G$--module homomorphism (the
$G$--action on $\mathbb Z H$ is $hx^i*h'=hx^ih'x^{-i}$). In
particular $H_1(\Gamma(H,\{z_i\}_{i\in\mathbb Z}))$ is a $\mathbb
Z G$--module and the Fox-derivative ${\cal F}\co  M(H,\{z_i\}_{i\in
\mathbb Z})\rightarrow H_1(\Gamma(H,\{z_i\}_{i\in \mathbb Z}))$ is
a $\mathbb Z G$--module isomorphism.

If $\alpha=\alpha_jx^j+\ldots+\alpha_{j+n}x^{j+n}\in ZH[x^{\pm 1}]$,
$\alpha_j$, $\alpha_{j+n}$ both not zero, then we call $n$ the length
of $\alpha$, $n=l(\alpha)$. Note that if $\alpha=\beta \gamma$ then
$l(\alpha)=l(\beta)+l(\gamma)$. This uses the fact that, because $G$
is a torsion-free 1--relator group, the group ring $\mathbb Z G$ has no
zero divisors (see Brodskii \cite{Brodskii}, Howie \cite{Howie1,Howie2}).

\begin{lemma}\label{lemma3_2} 
Suppose $\alpha=\alpha_0e_0+\ldots+\alpha_ne_n$ is an element of
$H_1(\Gamma(H,\{z_i\}_{i\in\mathbb Z}))\subseteq \bigoplus_{i\in
\mathbb Z}\mathbb Z He_i$. Then there exist elements
$s_1,\ldots,s_m\in R=R(H,\{z_i\}_{i\in\mathbb Z})$ and\break $f_1,\ldots,f_m
\in F(\{z_i\}_{i\in\mathbb Z})$ such that each $s_j$, $j=1,\ldots,m$,
is a word in the letters $z_0,\ldots,z_n$ and ${\cal F}(\prod_{k=0}^m
f_ks_kf_k^{-1}[R,R])=\alpha$. \end{lemma}

\begin{proof} The cycle $\alpha$ is a sum of edges in
$\Gamma(H,\{z_i\}_{i\in\mathbb Z})$ with edge labels involving only
letters from  $\{ z_0,\ldots,z_n \}$. These edges can be arranged to
form closed edge paths $P_1,\ldots,P_m$. Choose paths
$Q_1,\ldots,Q_m$ such that $Q_i$ connects the vertex $1$ to a vertex
occurring in $P_i$. Reading off the edge labels on the path
$Q_iP_iQ_i^{-1}$, $i=1,\ldots,m$, gives a word of the form
$f_is_if_i^{-1}$, where $s_i\in R$ involves only letters from $\{
z_0,\ldots,z_n \}$, and ${\cal F}(\prod_{k=0}^m
f_ks_kf_k^{-1}[R,R])=\alpha$. \end{proof}

\begin{proof}[Proof of \fullref{thm1_1}] Suppose $M(G,\{x,z\})$
is generated (as $\mathbb Z G$--module) by a single element. Then so
is $M(H,\{z_i\}_{i\in\mathbb Z})$. Hence
$H_1(\Gamma(H,\{z_i\}_{i\in\mathbb Z}))$ is also singly generated,
say by $\alpha$. Note that $z_1^2z_0^{-3} \in R(H, \{z_i\}_{i\in
\mathbb Z})$, so $\beta={\cal F}(z_1^2z_0^{-3}[R,R])$ is an element
of length one in $H_1(\Gamma(H,\{z_i\}_{i\in\mathbb Z}))$ and
$\beta=\gamma\alpha$ for some $\gamma\in \mathbb Z G=\mathbb Z
H[x^{\pm 1}]$. Since $l(\beta)=l(\gamma)+l(\alpha)$, we conclude
that the length of $\alpha$ is less or equal to one, so we may
assume $\alpha=\alpha_0e_0+\alpha_1e_1$. It follows from \fullref{lemma3_2} that there are elements
$s_1(z_0,z_1)$,\ldots,$s_m(z_0,z_1)$ in $F(z_0,z_1)$ that give rise
to $\mathbb Z G$--module generators for $M(H,\{z_i\}_{i\in\mathbb
Z})$. Thus the set ${\bf s}(z, xzx^{-1})=\{ s_0(z,
xzx^{-1}),\ldots,s_m(z, xzx^{-1})\}$ generates the $\mathbb Z
G$--module $M(G, \{ x, z \})$. \fullref{prop3_1} implies that
the kernel of the map $$\bar G=\langle x, z \ |\ {\bf s}(z,xzx^{-1})
\rangle \rightarrow G$$ that sends $x$ to $x$ and $z$ to $z$ is
perfect. This contradicts \fullref{lemma2_2}. \end{proof}

\section{Topological applications}

In the last section we have seen that the relation module
$M=M(G,\{x,y^4\})$ for the group $G=\langle x, y \ |\
xy^2x^{-1}=y^3\rangle$ is not generated by a single element, hence
is certainly not isomorphic to $\mathbb Z G$. In this section we
will use this fact to exhibit 2--complexes with fundamental group $G$
and the same Euler-characteristic that are not homotopically
equivalent.

Let $X$ be a CW--complex which is the union of a family of non-empty
subcomplexes $X_{\alpha}$, where $\alpha$ ranges over some index set
$J$. Let $\cal{N}$ be the nerve associated with this covering of
$X$. The nerve $\cal{N}$ is a simplicial complex with vertex set
$J$. The $n$--simplices are subsets $\{ \alpha_1,...,\alpha_n \}$,
$\alpha_i\in J$, $i=1,...,n$, such that the intersection
$X_{\alpha_1}\cap ... \cap X_{\alpha_n}$ is not empty.

The following result is an immediate consequence of the
Mayer--Vietoris spectral sequence (see Brown \cite[page 166]{Brown}).

\begin{lemma}\label{lemma4 1} If ${\cal N}$ is a tree we have a
long exact sequence
$$...\rightarrow H_{p+1}(X)\rightarrow \bigoplus_{\{\alpha, \alpha'\}\in
{\cal N}^{(1)}}H_p(X_{\alpha}\cap X_{\alpha'})\rightarrow
\bigoplus_{\alpha\in {\cal N}^{(0)}}H_p(X_{\alpha})\rightarrow
H_p(X)\rightarrow ...$$
\end{lemma}

Consider the amalgamated product of groups $G=G_1*_H G_2$. Let
$K_i$ be an Eilenberg--MacLane complex for $G_i$, $i=1,2$, and $L$
be an Eilenberg--MacLane complex for $H$. For convenience we assume
that each of these complexes have a single vertex. Let $u$, $v$,
$w$ be the vertices of $K_1$, $K_2$, $L$, respectively. An
Eilenberg--MacLane complex $K=K_1\cup (L\times [0,1])\cup K_2$ for
$G$ is obtained from $K_1$, $L\times [0,1]$, and $K_2$ by gluing
$L\times \{0\}$ to $K_1$ via a map induced by the inclusion
$H\hookrightarrow G_1$ and gluing $L\times \{1\}$ to $K_2$ via a
map induced by the inclusion $H\hookrightarrow G_2$.

\begin{thm}\label{thm4_2} \
{\rm(a)}\qua For $n > 2$ there is a short exact sequence
$$0\rightarrow \mathbb Z G \otimes _{G_1}\pi_n(K_1^{(n)})\oplus \mathbb Z G \otimes
_{G_2}\pi_n(K_2^{(n)})\rightarrow \pi_n(K^{(n)})\rightarrow \mathbb
Z G\otimes _H \pi_{n-1}(L^{(n-1)})\rightarrow 0.$$ 

{\rm(b)}\qua For $n=2$
there is a short exact sequence
$$0\rightarrow \mathbb Z G \otimes
_{G_1}\pi_2(K_1^{(2)})\oplus \mathbb Z G \otimes
_{G_2}\pi_2(K_2^{(2)}) \rightarrow \pi_2(K^{(2)})\rightarrow \mathbb
Z G\otimes _H M(H,{\bf x})\rightarrow 0,$$ where ${\bf x}$ is the
generating set for $H$ coming from the 1--skeleton of $L$.
\end{thm}

\begin{proof} Let $X$ be the $n$--skeleton of the universal covering $\tilde
K$ and let $p\co  \tilde K\rightarrow K$ be the covering projection.
Let $K_u$ be the $n$--skeleton of $K_1\cup (L\times [0,\frac{1}{2}])$
and $K_v$ be the $n$--skeleton of the other half, $K_2\cup (L\times
[\frac{1}{2},1])$. Note that $K_u \cap K_v=L^{(n-1)}\times
\{\frac{1}{2}\}$. Let $\tilde w$ be a fixed lift of the point
$\{w\}\times \{\frac{1}{2}\}$. Let $X_u$ be the component of
$p^{-1}(K_u)$ that contains $\tilde w$ and $X_v$ be the component of
$p^{-1}(K_v)$ that contains $\tilde w$. The intersection $Y=X_u\cap
X_v$ is a component of $p^{-1}(K_u\cap K_v)=p^{-1}(L^{(n-1)}\times
\{\frac{1}{2}\})$ and hence is homeomorphic to the $(n-1)$--skeleton
of the universal covering of $L$. Note that $p^{-1}(K_u)$ is the
disjoint union of the translates $gX_u$, where $g$ is taken from
$T(G/G_1)$, a transversal for $G/G_1$. Analogously, $p^{-1}(K_v)$ is
the disjoint union of the translates $g'X_u$, where $g'$ is taken
from $T(G/G_2)$. Let $\cal{N}$ be the nerve associated with the
covering of $X$ by the components of $p^{-1}(K_u)$ and
$p^{-1}(K_v)$. This nerve is a simplicial tree, isomorphic to the
Bass--Serre tree associated with the amalgamated product $G_1*_HG_2$.
The vertices of this tree are the translates $gX_u$ for $g\in
T(G/G_1)$ and $g'X_v$ for $g'\in T(G/G_2)$. The intersection
$gX_u\cap g'X_v \ne \emptyset$ if and only if there is a $g''\in G$
so that $gX_u=g''X_u$ and $g'X_v=g''X_v$. In this case $gX_u\cap
g'X_v=g''(X_v\cap X_v)=g''Y$. Hence the edges of ${\cal N}$ are the
translates $g''Y$, $g''\in T(G/H)$. We apply the previous lemma and
obtain a long exact sequence
\begin{align*}
\cdots\rightarrow H_{n+1}(X)\rightarrow
\hspace{-10pt} &\bigoplus_{g''\in T(G/H)}\hspace{-10pt}
H_n(g''Y)\rightarrow\\
\rightarrow\hspace{-10pt}&\bigoplus_{g\in T(G/G_1)}\hspace{-10pt}
H_n(gX_u)
\hspace{-10pt}\bigoplus_{g'\in T(G/G_2)}\hspace{-10pt} 
H_n(g'X_v)\rightarrow H_{n}(X)\rightarrow\cdots 
\end{align*}
Since $Y$ is $(n-1)$--dimensional and both
$X_u$, $X_v$ are $n$--dimensional and $(n-1)$--connected, this yields the
short exact sequence 
$$ 0 \rightarrow\hspace{-10pt} \bigoplus_{g\in T(G/G_1)}\hspace{-10pt}
H_n(gX_u)\oplus
\hspace{-10pt} \bigoplus_{g'\in T(G/G_2)}\hspace{-10pt} 
H_n(g'X_v)\rightarrow H_{n}(X)
\rightarrow\hspace{-10pt} \bigoplus_{g''\in T(G/H)}\hspace{-10pt}
H_{n-1}(g''Y)\rightarrow 0.$$ 
Let us first assume that $n>2$.
Since $X_u$ is the $n$--skeleton of the universal covering of the
Eilenberg--MacLane complex $K_u$ (which is homotopically equivalent
to $K_1$) we have $H_n(X_u)\approx \pi_n(X_u)\approx
\pi_n(K_1^{(n)})$ by the Hurewicz Theorem. So
$$\bigoplus_{g\in T(G/G_1)}H_n(gX_u)\approx \mathbb Z
G\otimes_{G_1}\pi_n(K_1^{(n)})$$ as $\mathbb Z G$--modules. By
analogous arguments we have 
\begin{align*}
&\bigoplus_{g'\in
T(G/G_2)}H_n(g'X_v)\approx \mathbb Z
G\otimes_{G_2}\pi_n(K_2^{(n)})\\
\tag*{\hbox{and}}
&\bigoplus_{g''\in
T(G/H)}H_{n-1}(g''Y)\approx \mathbb Z
G\otimes_{H}\pi_{n-1}(L^{(n-1)}).
\end{align*}
So the above short exact
sequence, after making the isomorphic replacements, yields the short
exact sequence exhibited in statement (a).

Let us assume now that $n=2$. In that case $Y$ is the 1--skeleton of
the universal covering of $L$ and hence is the Cayley-graph of the
group $H$, associated with the generating set ${\bf x}$ that arises
from the 1--skeleton of $L$ (recall that we assumed $L$ to have a
single vertex, so the 1--skeleton is a wedge of circles). Thus
$H_1(Y)\approx H_1(\Gamma(H,{\bf x}))\approx M(H,{\bf x})$ and
$$\bigoplus_{g''\in T(G/H)}H_1(g''Y)\approx \mathbb Z G \otimes_H
M(H,{\bf x}).$$ The above short exact sequence, after making the
isomorphic replacements, yields the short exact sequence exhibited
in statement (b). \end{proof}

A 2--complex $K$ is {\sl aspherical} if $\pi_2(K)=0$. A group is {\sl
aspherical} if it is the fundamental group of an aspherical
2--complex. A consequence of the above theorem is that for aspherical
groups every relation module is also a second homotopy module.
Indeed, assume $J$ is an aspherical 2--complex with fundamental group
$G$ (in particular $J$ is an Eilenberg--MacLane complex). Suppose
$M=M(G,{\bf x})$ is a relation module for $G$ associated with some
generating set. Let $L$ be an Eilenberg--MacLane complex for $G$ with
1--skeleton a bouquet of circles in one-to-one correspondence with
the elements of ${\bf x}$. Writing $G$ as an amalgamated product
$G=G*_G G$ we apply the above construction and build an
Eilenberg--MacLane complex $K=K_1\cup (L\times [0,1])\cup K_2$ with
$K_1=J$, $K_2=J$. The ends $L\times \{0\}$ and $L\times \{1\}$ are
attached to the two copies of $J$ via a map induced by the identity
map from $G$ to $G$. \fullref{thm4_2}(b) implies the following
result.

\begin{corollary}\label{cor4_2} Let $J$ be an aspherical 2--complex with fundamental group $G$. Let
$M(G,{\bf x})$ be a relation module associated with some generating
set ${\bf x}$ and let $L$ be an Eilenberg--MacLane complex with
1--skeleton a bouquet of circles in one-to-one correspondence with
${\bf x}$. Then the 2--complex $K^{(2)}=(J\cup(L\times [0,1])\cup
J)^{(2)}$ has Euler-characteristic
$\chi(K^{(2)})=2\chi_{min}(G)-1+|{\bf x}|$ and the second homotopy
module $\pi_2(K^{(2)})$ is isomorphic to the relation module
$M(G,{\bf x})$.
\end{corollary}

Consider the group $G$ presented by $\langle x, y \ |\
xy^2x^{-1}=y^3\rangle$. Let $J$ be the 2--complex built from this
presentation.  Since $J$ is aspherical (see Lyndon \cite{Lyndon}) we have
$\chi_{min}(G)$ $=\chi(J)=0$ and $J$ is the only 2--complex (up to
homotopy equivalence) on the minimal Euler-characteristic level.
The next result shows that the situation is different on the next
level, $\chi_{min}(G)+1$.

\begin{thm}\label{thm4_3} Let $J$ be the 2--complex built on the presentation
$$\langle x, y \ |\ xy^2x^{-1}=y^3\rangle$$ for $G$. Let $K$ be the
2--complex built on the presentation $$\langle x, y, x', y' \ |\
xy^2x^{-1}=y^3, x'{y'}^2{x'}^{-1}={y'}^3, x=x', y^4={y'}^4 \rangle
$$ for $G$. Then $J\vee S^2$ and $K$ are homotopically distinct
2--complexes with fundamental group $G$ and Euler-characteristic
$\chi_{min}(G)+1$.
\end{thm}

\begin{proof} We write $G$ as an amalgamated product $G=G*_G G$ and build
an Eilenberg--MacLane complex $K'=J\cup (L\times [0,1])\cup J$, where
we take for $L$ an Eilenberg--MacLane complex with 1--skeleton a
bouquet of two circles corresponding to the generating set $\{ x,
y^4\}$. By \fullref{cor4_2}, $\pi_2(K'^{(2)})$ is isomorphic to
$M(G,\{x,y^4\})$ and hence is different from $\mathbb Z G$. Note
that $K'$ can be built so that, after collapsing a maximal tree in
the 1--skeleton of $K'^{(2)}$ (which consists of a single edge) we
obtain the complex $K$. Since $\pi_2(J\vee S^2)=\mathbb Z G$ the
desired result follows. \end{proof}

In \cite{BerridgeDunwoody} P\,H Berridge and M Dunwoody
give an infinite sequence of pairwise distinct stably free
non-free relation modules for the trefoil group $G$ and ask the
question whether this can be used to construct infinitely many
pairwise non-homotopic 2--complexes with identical
Euler-characteristic and fundamental group $G$. The above theorem
answers this question affirmatively. Consider the trefoil group
presented by $\langle x, y \ | \ x^2=y^3 \rangle$. Let ${\bf
x_i}=\{x^{2i+1}, y^{3i+1}\}$, $i\in \mathbb N$. The set of
relation modules $\{ M(G,{\bf x_i}) \ |\ i\in \mathbb N \}$
contains infinitely many non-isomorphic stably free non-free
modules of rank one \cite{BerridgeDunwoody}.

\begin{thm}\label{thm4_4} Let $G$ be the trefoil group presented by
$\langle x, y \ |\ x^2=y^3 \rangle$. Let $K_i$ be the 2--complex
built on the presentation $$\langle x, y, x', y' \ |\ x^2=y^3,
{x'}^2={y'}^3, x^{2i+1}={x'}^{2i+1}, y^{3i+1}={y'}^{3i+1} \rangle
$$ for $G$. The set $\{ K_i \ |\ i\in \mathbb N \}$ contains infinitely many
homotopically distinct 2--complexes with fundamental group $G$ and
Euler-characteristic equal to one.
\end{thm}

In \cite{Lustig} M Lustig constructs an infinite
collection of 2--dimensional homotopy types with the same
fundamental group  and Euler characteristic. However the
fundamental group in these examples is distinct from the trefoil
group. \fullref{thm4_4} answers the precise question raised by
Berridge and Dunwoody \cite{BerridgeDunwoody} in the affirmative.

\bibliographystyle{gtart}
\bibliography{link}

\end{document}